\documentclass[12pt,a4paper]{amsart}

\usepackage{jipam}
\usepackage{amsmath,amssymb,amscd}
\usepackage{subfigure, epsfig}
\newcommand{\mb}{\mathbf}

\newcommand{\mc}{\mathcal}

\newcommand{\ds}{\displaystyle}

\theoremstyle{plain}

\numberwithin{equation}{section}

\title[A Trace Inequality]{A generalization of a trace inequality for positive definite matrices.}
\author{Elena-Veronica Belmega}
\address{Universit\'e Paris-Sud XI\\ SUPELEC\\ Laboratoire des signaux et syst\`emes\\ Gif-sur-Yvette, France. }
\email{belmega@lss.supelec.fr}
\urladdr{http://veronica.belmega.lss.supelec.fr}

\author{M. Jungers}
\address{CNRS\\ ENSEM\\ CRAN \\ Vandoeuvre, France. }
\email{marc.jungers@cran.uhp-nancy.fr}
\urladdr{http://perso.ensem.inpl-nancy.fr/Marc.Jungers/}

\author{Samson Lasaulce}
\address{CNRS\\ SUPELEC \\ Laboratoire des signaux et syst\`emes\\ Gif-sur-Yvette, France. }
\email{lasaulce@lss.supelec.fr}
\urladdr{http://samson.lasaulce.lss.supelec.fr}


\begin{document}

\begin{abstract}
In this note, we provide the generalization of the trace inequality derived in \cite{belmega-jipam-2009}. More precisely, we prove that \newline
$\mathrm{Tr} \left\{ \ds{\sum_{k=1}^K}(\mb{A}_k-\mb{B}_k) \left[\left(\ds{\sum_{\ell=1}^k}\mb{B}_{\ell}\right)^{-1}-
\left(\ds{\sum_{\ell=1}^k}\mb{A}_{\ell}\right)^{-1}\right] \right\} \geq 0$, for arbitrary $K \geq 1$ where $\mathrm{Tr}(\cdot)$ denotes the
matrix trace operator, $\mb{A}_1$, $\mb{B}_1$ are any positive definite matrices and $\mb{A}_k$, $\mb{B}_k$, for all $k \in \{2,\hdots,K\}$, are
any positive semidefinite matrices.
\end{abstract}

\maketitle

\section{Introduction}
\label{sec1}

Trace inequalities are useful in many areas such as the multiple input multiple output (MIMO) systems in control theory and communications. The
trace inequality derived in this paper is used to prove the sufficient condition that guarantees the uniqueness of a Nash equilibrium in certain
MIMO communications games with $K\geq 1$ players (see \cite{belmega-springer-2009} for details). To be more specific, the diagonally strict
concavity condition of Rosen \cite{rosen-eco-1965} is proven to be satisfied in the scenario of \cite{belmega-springer-2009}. The trace
inequality under discussion has already been proven in two special cases: in \cite{abadir-book-2005} for $K=1$ and in \cite{belmega-jipam-2009}
for $K=2$. In what follows, we will provide the proof for the general case where $K \geq 1$ is arbitrary.

\begin{theorem}
\label{lemma_K} Let $K$ be a strictly positive integer, $\mb{A}_1$, $\mb{B}_1$ be any positive definite matrices and $\forall k \in \{2,\hdots,
K\}$, $\mb{A}_k$, $\mb{B}_k$ be any positive semidefinite matrices. Then
\begin{equation}
\label{trace_K} \mc{T}_K\triangleq \mathrm{Tr} \left\{ \ds{\sum_{k=1}^K}(\mb{A}_k-\mb{B}_k)
\left[\left(\ds{\sum_{\ell=1}^k}\mb{B}_{\ell}\right)^{-1}- \left(\ds{\sum_{\ell=1}^k}\mb{A}_{\ell}\right)^{-1}\right] \right\} \geq  0,
\end{equation}
where $\mathrm{Tr}(\cdot)$ denotes the matrix trace operator.
\end{theorem}

\section{Auxiliary Results}
\label{sec2}

In order to prove Theorem \ref{lemma_K}, we will use the following auxiliary results.

\begin{lemma} \cite{belmega-jipam-2009}
\label{lemma_3} Let $\mb{A}$, $\mb{B}$ be two positive definite matrices and $\mb{C}$, $\mb{D}$ be two positive semidefinite matrices and
$\mb{X}$ a Hermitian matrix. Then
\begin{equation}
\mathrm{Tr} \left\{\mb{X}\mb{A}^{-1}\mb{X}\mb{B}^{-1}\right\} - \mathrm{Tr} \left\{\mb{X}(\mb{A}+\mb{C})^{-1}\mb{X}(\mb{B}+\mb{D})^{-1}\right\}
\geq 0.
\end{equation}
\end{lemma}

The proof can be found in \cite{belmega-jipam-2009}.

\begin{lemma}
\label{lemma_4} Let $\mb{A}$, $\mb{B}$ be two positive definite matrices, $\mb{C}$, $\mb{D}$, two positive semi-definite matrices. Then
\begin{equation}
\begin{array}{lcl}
\mathrm{Tr} \left\{(\mb{A}-\mb{B})(\mb{B}+\mb{D})^{-1}(\mb{C}-\mb{D})(\mb{A}+\mb{C})^{-1}\right\} & = &  \\
\mathrm{Tr} \left\{(\mb{C}-\mb{D})(\mb{B}+\mb{D})^{-1}(\mb{A}-\mb{B})(\mb{A}+\mb{C})^{-1}\right\} & \in & \mathbb{R}.
\end{array}
\end{equation}
\end{lemma}

\begin{proof}
To prove this result, let us define $\mc{E}$ as follows:
\begin{equation}
\label{eq:e} \mc{E}=\mathrm{Tr} \left\{(\mb{C}-\mb{D})\left[(\mb{B}+\mb{D})^{-1}-(\mb{A}+\mb{C})^{-1}\right]\right\}
\end{equation}
We observe that $\mc{E}$ can be written in two different ways:
\begin{equation}
\label{eq1}
\begin{array}{lcl}
\mc{E}& = & \mathrm{Tr} \left\{(\mb{C}-\mb{D})(\mb{B}+\mb{D})^{-1}[\mb{A}+\mb{C}-\mb{B}-\mb{D}](\mb{A}+\mb{C})^{-1}\right\} \\
& = & \mathrm{Tr} \left\{(\mb{C}-\mb{D})(\mb{B}+\mb{D})^{-1}(\mb{C}-\mb{D})(\mb{A}+\mb{C})^{-1}\right\} + \\
& & \mathrm{Tr} \left\{(\mb{C}-\mb{D})(\mb{B}+\mb{D})^{-1}(\mb{A}-\mb{B})(\mb{A}+\mb{C})^{-1}\right\},
\end{array}
\end{equation}
and
\begin{equation}
\label{eq2}
\begin{array}{lcl}
\mc{E}& = & \mathrm{Tr} \left\{(\mb{C}-\mb{D})(\mb{A}+\mb{C})^{-1}[\mb{A}+\mb{C}-\mb{B}-\mb{D}](\mb{B}+\mb{D})^{-1}\right\} \\
& = & \mathrm{Tr} \left\{(\mb{C}-\mb{D})(\mb{A}+\mb{C})^{-1}(\mb{C}-\mb{D})(\mb{B}+\mb{D})^{-1}\right\} +\\
& & \mathrm{Tr} \left\{(\mb{C}-\mb{D})(\mb{A}+\mb{C})^{-1}(\mb{A}-\mb{B})(\mb{B}+\mb{D})^{-1}\right\}.
\end{array}
\end{equation}
Using this fact and the commutative property of the trace of a matrix product, the desired result follows directly. The only thing left to be
proven is that $\mc{E}$ is real. To this end, if we denote by $\mb{M} = (\mb{C}-\mb{D})(\mb{B}+\mb{D})^{-1}(\mb{A}-\mb{B})(\mb{A}+\mb{C})^{-1}
$, we observe that $\mb{M}^H= (\mb{A}+\mb{C})^{-1}(\mb{A}-\mb{B})(\mb{B}+\mb{D})^{-1}(\mb{C}-\mb{D})$. Therefore, we obtain that
$\mathrm{Tr}(\mb{M}^H)=\mathrm{Tr}(\mb{M})$ and also $\mathrm{Tr}(\mb{M}) \in \mathbb{R}$.
\end{proof}

\section{Proof of Theorem \ref{lemma_K}}
\label{sec3}

Define, for all $k\geq 1$, $\mb{X}_k=\ds{\sum_{i=1}^k \mb{A}_i}$ and $\mb{Y}_k=\ds{\sum_{i=1}^k \mb{B}_i}$. Notice that $\mb{X}_k$ and
$\mb{Y}_k$ are positive definite matrices. We observe that $\mc{T}_K$ can be re-written recursively as follows:
\begin{equation}
\label{recursion} \left\{
\begin{array}{lcl}
\mc{T}_1 & =& \mathrm{Tr} \left\{(\mb{A}_1-\mb{B}_1)\mb{Y}_1^{-1}(\mb{A}_1-\mb{B}_1)\mb{X}_1^{-1}\right\}\\
\\
\mc{T}_K & = & \mc{T}_{K-1} + \mathrm{Tr} \left\{(\mb{A}_K-\mb{B}_K)\mb{Y}_K^{-1}(\mb{A}_K-\mb{B}_K)\mb{X}_K^{-1}\right\} +  \\
\\
& & \mathrm{Tr} \left\{(\mb{A}_K-\mb{B}_K)\mb{Y}_K^{-1}(\mb{X}_{K-1}- \mb{Y}_{K-1})\mb{X}_K^{-1}\right\}
\end{array} \right.
\end{equation}
We proceed in two steps. First, we find a lower bound for $\mc{T}_K$ and then we prove that this bound is positive.

We start by proving that, for all $K \geq 1$:
\begin{equation}
\label{lower_bound} \mc{T}_K \geq \frac{1}{2} \ds{\sum_{i=1}^K}
\mathrm{Tr}\left\{(\mb{A}_i-\mb{B}_i)\mb{Y}_i^{-1}(\mb{A}_i-\mb{B}_i)\mb{X}_i^{-1}\right\} + \frac{1}{2}
\mathrm{Tr}\left\{(\mb{X}_K-\mb{Y}_K)\mb{Y}_K^{-1}(\mb{X}_K-\mb{Y}_K)\mb{X}_K^{-1}\right\}
\end{equation}
To this end we proceed by induction on $K$. For all $K \geq 1$, define the proposition:
\begin{equation}
\mc{P}_K: \ \mc{T}_K \geq  \frac{1}{2} \ds{\sum_{i=1}^K}
\mathrm{Tr}\left\{(\mb{A}_i-\mb{B}_i)\mb{Y}_i^{-1}(\mb{A}_i-\mb{B}_i)\mb{X}_i^{-1}\right\} + \frac{1}{2}
\mathrm{Tr}\left\{(\mb{X}_K-\mb{Y}_K)\mb{Y}_K^{-1}(\mb{X}_K-\mb{Y}_K)\mb{X}_K^{-1}\right\}.
\end{equation}
It is easy to check that, for $K=1$, $\mc{P}_1$ is true:
\begin{equation}
\begin{array}{lcl}
\mc{T}_1 & = & \mathrm{Tr} \left\{(\mb{A}_1-\mb{B}_1)\mb{Y}_1^{-1}(\mb{A}_1-\mb{B}_1)\mb{X}_1^{-1}\right\} \\
\\
& = & \frac{1}{2}\mathrm{Tr} \left\{(\mb{A}_1-\mb{B}_1)\mb{Y}_1^{-1}(\mb{A}_1-\mb{B}_1)\mb{X}_1^{-1}\right\} +\frac{1}{2}\mathrm{Tr}
\left\{(\mb{X}_1-\mb{Y}_1)\mb{Y}_1^{-1}(\mb{X}_1-\mb{Y}_1)\mb{X}_1^{-1}\right\}.
\end{array}
\end{equation}
Now, let us assume that $\mc{P}_{K-1}$ is true and prove that $\mc{P}_K$ is also true. We have that:

\begin{equation}
\label{eq3}
\begin{array}{lcl}
\mc{T}_{K-1} & \geq & \frac{1}{2} \ds{\sum_{i=1}^{K-1}}
\mathrm{Tr}\left\{(\mb{A}_i-\mb{B}_i)\mb{Y}_{i}^{-1}(\mb{A}_i-\mb{B}_i)\mb{X}_{i}^{-1}\right\} + \\
\\
& & \frac{1}{2} \mathrm{Tr}\left\{(\mb{X}_{K-1}-\mb{Y}_{K-1})\mb{Y}_{K-1}^{-1}(\mb{X}_{K-1}-\mb{Y}_{K-1})\mb{X}_{K-1}^{-1}\right\}.
\end{array}
\end{equation}

From (\ref{eq3}) and the recursive formula (\ref{recursion}), we further obtain:

\begin{equation}
\begin{array}{lcl}
\mc{T}_K & \geq & \frac{1}{2} \ds{\sum_{i=1}^{K-1}}
\mathrm{Tr}\left\{(\mb{A}_i-\mb{B}_i)\mb{Y}_{i}^{-1}(\mb{A}_i-\mb{B}_i)\mb{X}_{i}^{-1}\right\} + \\
\\
& & \frac{1}{2} \mathrm{Tr}\left\{(\mb{X}_{K-1}-\mb{Y}_{K-1})\mb{Y}_{K-1}^{-1}(\mb{X}_{K-1}-\mb{Y}_{K-1})\mb{X}_{K-1}^{-1}\right\} + \\
\\
& &
\mathrm{Tr} \left\{(\mb{A}_K-\mb{B}_K)\mb{Y}_K^{-1}(\mb{A}_K-\mb{B}_K)\mb{X}_K^{-1}\right\} + \mathrm{Tr} \left\{(\mb{A}_K-\mb{B}_K)\mb{Y}_K^{-1}(\mb{X}_{K-1}- \mb{Y}_{K-1})\mb{X}_K^{-1}\right\} \\
\\
 & \stackrel{(a)}{\geq} & \frac{1}{2} \ds{\sum_{i=1}^{K-1}} \mathrm{Tr}\left\{(\mb{A}_i-\mb{B}_i)\mb{Y}_{i}^{-1}(\mb{A}_i-\mb{B}_i)\mb{X}_{i}^{-1}\right\} + \\
\\
 & & \frac{1}{2} \mathrm{Tr}\left\{(\mb{X}_{K-1}-\mb{Y}_{K-1})\mb{Y}_{K}^{-1}(\mb{X}_{K-1}-\mb{Y}_{K-1})\mb{X}_{K}^{-1}\right\}+ \\
\\
& & \mathrm{Tr} \left\{(\mb{A}_K-\mb{B}_K)\mb{Y}_K^{-1}(\mb{A}_K-\mb{B}_K)\mb{X}_K^{-1}\right\} + \mathrm{Tr} \left\{(\mb{A}_K-\mb{B}_K)\mb{Y}_K^{-1}(\mb{X}_{K-1}- \mb{Y}_{K-1})\mb{X}_K^{-1}\right\} \\
\\
& \stackrel{(b)}{=} & \frac{1}{2} \ds{\sum_{i=1}^{K-1}} \mathrm{Tr}\left\{(\mb{A}_i-\mb{B}_i)\mb{Y}_{i}^{-1}(\mb{A}_i-\mb{B}_i)\mb{X}_{i}^{-1}\right\} + \\
\\
& &\frac{1}{2} \mathrm{Tr}\left\{(\mb{X}_{K-1}-\mb{Y}_{K-1})\mb{Y}_{K}^{-1}(\mb{X}_{K-1}-\mb{Y}_{K-1})\mb{X}_{K}^{-1}\right\} +\\
\\
& & \mathrm{Tr} \left\{(\mb{A}_K-\mb{B}_K)\mb{Y}_K^{-1}(\mb{A}_K-\mb{B}_K)\mb{X}_K^{-1}\right\} + \frac{1}{2}\mathrm{Tr} \left\{(\mb{A}_K-\mb{B}_K)\mb{Y}_K^{-1}(\mb{X}_{K-1}- \mb{Y}_{K-1})\mb{X}_K^{-1}\right\} + \\
\\
& & \frac{1}{2} \mathrm{Tr} \left\{(\mb{X}_{K-1}-\mb{Y}_{K-1})\mb{Y}_K^{-1}(\mb{A}_{K}- \mb{B}_{K})\mb{X}_K^{-1}\right\} \\
\\
& = & \frac{1}{2} \ds{\sum_{i=1}^{K}} \mathrm{Tr}\left\{(\mb{A}_i-\mb{B}_i)\mb{Y}_{i}^{-1}(\mb{A}_i-\mb{B}_i)\mb{X}_{i}^{-1}\right\} + \\
\\
& & \frac{1}{2} \mathrm{Tr}\left\{(\mb{X}_{K-1}-\mb{Y}_{K-1})\mb{Y}_{K}^{-1}(\mb{X}_{K-1}-\mb{Y}_{K-1})\mb{X}_{K}^{-1}\right\}+ \\
\\
& & \frac{1}{2} \mathrm{Tr} \left\{(\mb{A}_K-\mb{B}_K)\mb{Y}_K^{-1}(\mb{A}_K-\mb{B}_K)\mb{X}_K^{-1}\right\} + \frac{1}{2}\mathrm{Tr} \left\{(\mb{A}_K-\mb{B}_K)\mb{Y}_K^{-1}(\mb{X}_{K-1}- \mb{Y}_{K-1})\mb{X}_K^{-1}\right\}  + \\
\\
& &  \frac{1}{2} \mathrm{Tr} \left\{(\mb{X}_{K-1}-\mb{Y}_{K-1})\mb{Y}_K^{-1}(\mb{A}_{K} - \mb{B}_{K})\mb{X}_K^{-1}\right\} \\
\end{array}
\end{equation}
\begin{equation}
\nonumber
\begin{array}{lcl}
& = &\frac{1}{2} \ds{\sum_{i=1}^{K}} \mathrm{Tr}\left\{(\mb{A}_i-\mb{B}_i)\mb{Y}_{i}^{-1}(\mb{A}_i-\mb{B}_i)\mb{X}_{i}^{-1}\right\} + \\
\\
& & \frac{1}{2} \mathrm{Tr}\left\{(\mb{X}_{K-1}+\mb{A}_K-\mb{Y}_{K-1}-\mb{B}_K)\mb{Y}_{K}^{-1}(\mb{X}_{K-1}+\mb{A}_K-\mb{Y}_{K-1}-\mb{B}_K)\mb{X}_{K}^{-1}\right\}\\
\\
& \stackrel{(c)}{=} & \frac{1}{2} \ds{\sum_{i=1}^{K}}
\mathrm{Tr}\left\{(\mb{A}_i-\mb{B}_i)\mb{Y}_{i}^{-1}(\mb{A}_i-\mb{B}_i)\mb{X}_{i}^{-1}\right\} +
\frac{1}{2} \mathrm{Tr}\left\{(\mb{X}_{K}-\mb{Y}_{K})\mb{Y}_{K}^{-1}(\mb{X}_{K}-\mb{Y}_{K})\mb{X}_{K}^{-1}\right\}.\\
\end{array}
\end{equation}

The inequality (a) follows by applying Lemma \ref{lemma_3} to the second term on the right and also by considering that
$\mb{X}_K=\mb{X}_{K-1}+\mb{A}_K$ and $\mb{Y}_K=\mb{Y}_{K-1}+\mb{B}_K$. The equality (b) follows from Lemma \ref{lemma_4}.

The last step of the proof of Theorem \ref{lemma_K} reduces to showing that the term on the right side of the equality (c) is positive. This can
be easily checked by observing that all the terms of the form $\mathrm{Tr}\left\{\mb{X}\mb{B}^{-1}\mb{X}\mb{A}^{-1}\right\}$, with $\mb{X}$ a
Hermitian matrix, $\mb{A}$ and $\mb{B}$ two  positive definite matrices, can be re-written as $\mathrm{Tr}(\mb{N}\mb{N}^H) \geq 0$ with
$\mb{N}=\mb{A}^{-1/2}\mb{X}\mb{B}^{-1/2}$. Thus, for all $K\geq 1$, $\mc{T}_K \geq 0$ and the desired result has been proven.

\vspace{10pt} \textbf{Acknowledgment.} The first author's work was partially supported by the French L'Or\'eal Program \emph{``For young women
doctoral candidates in science''} 2009.

\end{document}